\documentclass [12pt,reqno]{amsart}
\usepackage{ucs}
\usepackage{amsmath, amssymb, amscd}
\usepackage{hyperref}
\usepackage[applemac]{inputenc}
\usepackage{amsmath}
\usepackage{amssymb}
\usepackage{graphicx}
\usepackage[usenames,dvipsnames,svgnames,table]{xcolor}
\usepackage{enumitem} 
\usepackage{tikz-cd}
\usepackage{mathrsfs}

\usepackage[top=1.8in, bottom=1.6in, left=1.6in, right=1.6in]{geometry}

\newtheorem{theorem}{Theorem}[section]
\newtheorem{definition}[theorem]{Definition}
\newtheorem{lemma}[theorem]{Lemma}
\newtheorem{prop}[theorem]{Proposition}

\newtheorem{corollary}[theorem]{Corollary}

\newtheorem{remark}[theorem]{Remark}
\newtheorem*{intro-remark}{Remark}

\renewcommand{\epsilon}{\varepsilon}

\DeclareMathAlphabet{\mathpzc}{OT1}{pzc}{m}{it}
\usepackage{mathrsfs}

\newcommand{\N}{\mathbb{N}}

\newcommand{\C}{\mathbb{C}}

\newcommand{\LT}{\operatorname{LT}}

\newcommand{\Hom}{\text{Hom}}

\begin{document}
\thispagestyle{empty}
\title[A new algorithm for 3-sphere recognition]{A new algorithm for 3-sphere recognition}

\author{Michael Heusener}
\address{
Universit\'e Clermont Auvergne, Universit\'e Blaise Pascal, Laboratoire de Math\`ematiques, 
BP 10448, F-63000 Clermont-Ferrand \\ CNRS, UMR 6620, LM, F-63178 Aubiere, France}
\email{Michael.Heusener@math.univ-bpclermont.fr} 

\author{Raphael Zentner} 
\date{October 2016}
\address {Fakult\"at f\"ur Mathematik \\
Universit\"at Regensburg\\ 93040 Regensburg \\
Germany}
\email{raphael.zentner@mathematik.uni-regensburg.de}

\maketitle

\begin{abstract}
We prove the existence of a new algorithm for 3-sphere recognition based on Gr\"obner basis methods applied to the variety of $\text{\em SL}(2,\C)$-representation of the fundamental group. An essential input is a recent result of the second author, stating that any integer homology 3-sphere different from the 3-sphere admits an irreducible representation of its fundamental group in $\text{\em SL}(2,\C)$. This result, and hence our algorithm, build on the geometrisation theorem of 3-manifolds.
\end{abstract}

\section*{Introduction}
Rubinstein \cite{Rubinstein} has introduced an algorithm that recognises the 3-sphere, starting from a triangulation of the 3-manifold. His approach is based on normal surface theory, and it was simplified later by Thompson \cite{Thompson}. These algorithms have been established before the Poincar\'e conjecture and the geometrisation conjecture in dimension 3 were known to hold. Since then, the problem of 3-sphere recognition is equivalent to the recognition of the trivial group amongst fundamental groups of 3-manifolds. 

For instance, it is known that the following algorithm detects if a finite presentation $\langle \, S \, | \, R \, \rangle$ of the fundamental group $\pi$ of a 3-manifold represents the trivial group:

\begin{enumerate}
	\item Use two computers, Computer 1 and Computer 2. 
	\item On Computer 1, check successively whether there is a non-trivial homomorphism from $\pi$ to the symmetric group $S_n$, for $n=2, 3, \dots$. 
	\item On Computer 2, run the Todd-Coxeter algorithm applied to the trivial subgroup  $\{ 1 \} \subseteq \langle \, S \, | \, R \, \rangle$, see \cite{Neubueser,Annex-Neubueser}.
\end{enumerate} 
If the presentation represents a non-trivial group, then the program on Computer 1 will eventually  stop because 3-manifold groups are residually finite \cite{Hempel}. If the presentation represents the trivial group, the program on Computer 2 will eventually stop. 

Notice that the algorithm does stop because we suppose we are {\em guaranteed} that $\langle \, S \, | \, R \, \rangle$ is the presentation of the fundamental group of a 3-manifold. 

We suggest a new, somewhat simpler, and presumably more practical algorithm below. 

\section*{Acknowledgement} We wish to thank Martin Bridson, Stefan Friedl and Saul Schleimer for helpful conversation. 
The  authors are grateful for support by the SFB `Higher Invariants' at the University of Regensburg, funded by the Deutsche Forschungsgesellschaft (DFG).

\section{Review of Krull dimension, Hilbert polynomials, and Gr\"obner bases} 
By a (complex) affine algebraic variety we understand a subset of $\C^N$, for some $N \in \N$, which is the zero set of finitely many polynomials in the ring $R = \C[x_1, \dots, x_N]$. An affine algebraic variety is called irreducible if it is not the union of two non-empty strictly smaller varieties which are closed in the Zariski topology.

\begin{definition}
The Krull dimension of an affine algebraic variety $V$ is defined to be the maximal number $d$ such that there are irreducible sub-varieties $V_0, \dots, V_d$ which form a strictly increasing chain $V_0 \subseteq V_1 \subseteq \dots \subseteq V_d = V$. 
\end{definition}

We refer to \cite[Chapter 9, §3]{CLO} for the notions of Hilbert series and Hilbert polynomials of an ideal $I \subseteq \C[x_1, \dots, x_N]$. The following result is classical and can be found, for instance, in \cite[Section 5.6]{Kreuzer-Robbiano}. 

\begin{prop}\label{dimension and hilbert polynomial}
Let $V$ be an affine algebraic variety over $\C$, determined by the ideal $I \subseteq R$. Then the Krull dimension of $V$ is equal to the degree of the Hilbert polynomial of the ideal $I \subseteq R$,
\[
	\dim(V) = \deg(\text{HP}_{R/I}).
\]
\end{prop}

We fix some graded order on the monomials of $R$. For instance, this can be the graded lexicographical order. Then every element of $R$ has a well-determined leading term. 
Following standard notation, we denote by $\langle \LT(I) \rangle$ the ideal generated by the leading terms of the elements in $I$, see for instance \cite{CLO}. Recall that a Gr\"obner basis for $I$ associated to the chosen order is a finite subset of $I$ whose leading terms generate $\langle \LT(I) \rangle$. 

\begin{prop}\label{computing dimension}
	The following algorithm computes the Krull-dimension of an affine algebraic variety $V(I)$ determined by an ideal $I \subseteq R$. 
	\begin{enumerate}
		\item Compute a Gr\"obner basis for $I$ with respect to the chosen order.
		\item Consider the subsets $S \subseteq \{ x_1, \dots, x_n \}$ such that $x^{s}$, defined to be the product of the elements of $S$, does not lie in the monomial ideal $\langle \LT(I) \rangle$. Let $m$ denote the maximal cardinality of all these subsets $S$. 
		\item $\dim(V(I)) = \dim(V(\langle \LT(I) \rangle)) = m$. 
	\end{enumerate}
\end{prop}
\begin{proof}
We refer to \cite[Chapter 9, §1, Proposition 3]{CLO} for the proof that the second step determines the dimension of a variety associated to a monomial ideal such as $\langle \LT(I) \rangle$. The  following key observation is attributed to Macaulay, see \cite[Chapter 9, §3, Proposition 4]{CLO}: The Hilbert series (and hence the Hilbert polynomial) of the monomial ideal $\langle \LT(I) \rangle$ is equal to the Hilbert series of the monomial $I$, and hence we have equality of the associated Hilbert polynomials,
\[
	\text{\em HP}_{R/I} = \text{\em HP}_{R/\langle \LT(I) \rangle} \, . 
\]
 Therefore $\dim(V(I)) = m$ by Proposition \ref{dimension and hilbert polynomial}. 
\end{proof}

\section{The representation variety}

Let  $\pi$ be a finitely generated group, and let $ \langle s_1,\ldots, s_n \, | \, r_1, \dots, r_m \rangle$ be a presentation of $\pi$.
A $\text{\em SL}(2,\mathbb{C})$-repre\-sen\-ta\-tion of $\pi$ is a homomorphism 
$\rho\colon\pi\to\text{\em SL}(2,\mathbb{C})$.
\begin{definition} 
The 
$\text{\em SL}(2,\mathbb{C})$-representation variety is
\[
R(\pi) = \mathrm{Hom} (\pi,\text{\em SL}(2,\mathbb{C}))
\subseteq 
\text{\em SL}(2,\mathbb{C})^n\subseteq M(2,\C)^n\cong\C^{4n}\,.
\]
\end{definition}
The representation variety 
$R(\pi)$ is contained in $\text{\em SL}(2,\C)^n $ via the inclusion 
$\rho\mapsto\big(\rho(s_1),\ldots,\rho(s_n)\big)$, and
it is the set of solutions of a finite system of polynomial equations in the matrix coefficients (in fact, $4m+n$ many), hence it is an affine algebraic variety. 

%


\section{A new algorithm for detecting the trivial group among 3-manifold groups}

The following result has recently been established by the second author \cite{Zentner}. 
\begin{theorem}\label{existence_irr_rep}
	Let $Y$ be an integer homology 3-sphere different from the 3-sphere. Then there is an irreducible representation $\rho\colon \pi_1(Y) \to \text{SL}(2,\C)$. 
\end{theorem}


With this at hand, we are able to prove the following

\begin{theorem}\label{algorithm}
Let $\pi =  \langle s_1,\ldots, s_n \, | \, r_1, \dots, r_m \rangle$ be a presentation of the fundamental group of a 3-manifold with $n$ generators. Then the following algorithm decides whether or not $\pi$ is the trivial group.
\begin{enumerate}
	\item Abelianise $\pi$. If the abelianisation is non-trivial, $\pi$ isn't the trivial group.
	\item If the abelianisation $\pi_{\text{ab}}$ of $\pi$ is trivial, fix a graded monomial order in $\C[x_1, \dots, x_{4n}]$, and compute a Gr\"obner basis for the affine algebraic variety 
	\begin{equation*}
		R(\pi) = \Hom(\pi,\text{SL}(2,\C)) \subseteq \C^{4n}\,.
	\end{equation*} 
	\item From the Gr\"obner basis, determine if the Krull dimension $\dim(R(\pi))$ of $R(\pi)$ is equal to $0$ or bigger than $0$, following the algorithm in Proposition \ref{computing dimension} above. 
	\item If $\dim(R(\pi)) \neq 0$, then $\pi$ is not the trivial group. \\ 
	If $\dim(R(\pi)) = 0$, then $\pi$ is the trivial group. 
\end{enumerate}
\end{theorem}
\medskip 
\begin{remark}
 Theorem~\ref{algorithm} is in contrast to the following general fact: Whether or not a finite presentation represents the trivial group is undecidable, see \cite{Adyan1,Adyan2,Rabin} and for a survey \cite{Miller}.
\end{remark}

\begin{remark}
	In the preceding result, the presentation is not required to be geometrical (for instance, obtained from a Morse decomposition of a 3-manifold, or a triangulation.)
	However, we do require that the presentation is that of the fundamental group of a 3-manifold. In general, it is undecidable whether or not a given group is the fundamental group of a 3-manifold, see for instance the work of Groves, Manning, and Wilton \cite{GMW}, and the work Aschenbrenner, Friedl, Wilton for a survey \cite{AFW}. 
\end{remark}

The proof of this Theorem \ref{algorithm} will make use of the following lemma.

\begin{lemma}\label{dimension}
Let $\pi$ be a finitely generated group.
	If the representation variety $V(\pi)$ contains an irreducible representation, 
	then $\dim V(\pi) \geq 3$. 
\end{lemma}
\begin{proof}
Let $\rho\colon\pi\to \text{\em SL}(2,\C)$ be an irreducible representation. 
The group $\text{\em SL}(2,\C)$ acts by conjugation on the representation variety $R(\pi)$.
More precisely, for $A\in\text{\em SL}(2,\C)$  we define
$(A.\rho) (\gamma) = A \rho(\gamma) A^{-1}$ for all $\gamma\in\pi$, and
we let
$
 O(\rho)=\{A.\rho\mid A\in\text{\em SL}(2,\mathbb{C})\}
$ 
denote the orbit of $\rho$. Notice that the stabiliser of an irreducible representation is the centre of 
$\text{\em SL}(2,\C)$.

Now, Theorem~1.27 of \cite{LM85} implies that $O(\rho)\subset R(\pi)$ 
is a closed algebraic subset,
and Lemma 3.7 of \cite{Newstead1978} implies that $\dim O(\rho) =3$ since $\rho$ is irreducible.
Hence $3 = \dim O(\rho) \leq \dim R(\pi)$ by definition of the notion of Krull dimension.
\end{proof}

\begin{proof}[Proof of Theorem \ref{algorithm}]

If $\pi$ has trivial abelianisation and is not the fundamental group of the 3-sphere,  then there is an irreducible representation $\rho: \pi \to \text{\em SL}(2,\C)$ by Theorem \ref{existence_irr_rep}. By Lemma \ref{dimension}, we conclude that $R(\pi) = \Hom(\pi,\text{\em SL}(2,\C))$ has Krull dimension at least $3$. 
Hence if the Gr\"obner basis computation yields $\dim(V) = 0$, then $\pi$ must be the trivial group. 
\end{proof}

\section{A new algorithm for 3-sphere recognition}

We think of a 3-manifold $Y$ as being given by a Heegaard diagram. From this we can read off a presentation of the fundamental group. If $g$ is the genus of the Heegaard diagram, and if $k$ is the number of intersections in the Heegaard diagram (counted absolutely, and not up to sign), we obtain a presentation of the fundamental group $\pi_1(Y)$ of length $g+k$. 
\\
\begin{corollary}
	The combination of \begin{enumerate}
	\item
	the standard algorithm to pass from a Heegaard diagram of a 3-manifold $Y$ to a presentation $\pi =  \langle \, S \, | \, R \,\rangle$ of its  fundamental group together with 
	\item the algorithm of Theorem \ref{algorithm}
	\end{enumerate}
 is an algorithm that detects the 3-sphere.  
\end{corollary}


	The input data of this algorithm is given by a Heegaard diagram, and not by a triangulation, as it is the case in the Rubinstein-Thompson algorithm. This may turn out more practical in concrete cases. In fact, we only need a presentation of the fundamental group. 
	
	Furthermore, any triangulation comes with a canonical Heegaard diagram, and a theorem of Reidemeister and Singer states that any two Heegaard diagrams of the same 3-manifold are stably equivalent 
	(see \cite{Fomenko-Matveev}). However, it still seems unknown how much the Heegaard genus of a 3-manifold can differ from the Heegaard genus of a diagram coming from a minimal triangulation, and how many stabilisations/destabilisations one needs to pass from one to the other. From this point of view, it may be that our algorithm uses essentially smaller input data than the previously mentioned one.

\section{Complexity questions}
For the notion of complexity classes such as $\mathsf{NP}$ we refer to \cite{Garey-Johnson}. 

\subsection{Our algorithm}

The question whether finite systems of polynomial equations define algebraic varieties of dimension greater or equal to $d$ is $\mathsf{NP}$-hard for any $d \geq 0$ by a result of Koiran, see \cite[Proposition~1.1]{Koiran}.

In our situation, we know a few more facts about the varieties $R(\pi)$ in question. For instance, these are always determined by polynomial equations with integer coefficients. Furthermore, $R(\pi)$ always contains the trivial representation, and therefore this variety is always non-empty. 

It is unclear to us whether these facts decrease the complexity of our algorithm, but given Koiran's result, we rather expect the algorithm not to be of polynomial length in terms of the input size, and hence not to lie in the complexity class $\mathsf{P}$. This also seems consistent with some numerical evidence we have obtained. 

\subsection{$\mathsf{NP}$ algorithms}

Schleimer has shown that the 3-sphere recognition problem lies in the complexity class $\mathsf{NP}$, i.e. there is a non-deterministic algorithm that detects the 3-sphere in polynomial time \cite{Schleimer}. Kuperberg has shown that the unknot detection problem lies in the complexity class $\mathsf{coNP}$, provided the generalised Riemann hypothesis (GRH) holds. Based on Kuperberg's approach, the second author has shown in \cite{Zentner} that the 3-sphere recognition problem lies in the complexity class $\mathsf{coNP}$ modulo GRH.

\bibliographystyle{siam} 
\bibliography{3-sphere_recognition}

\def\cprime{$'$}
\begin{thebibliography}{10}

\bibitem{Adyan2}
{\sc S.~I. Adyan}, {\em Finitely presented groups and algorithms}, Dokl. Akad.
  Nauk SSSR (N.S.), 117 (1957), pp.~9--12.

\bibitem{Adyan1}
\leavevmode\vrule height 2pt depth -1.6pt width 23pt, {\em Unsolvability of
  some algorithmic problems in the theory of groups}, Trudy Moskov. Mat. Ob\v
  s\v c., 6 (1957), pp.~231--298.

\bibitem{AFW}
{\sc M.~Aschenbrenner, S.~Friedl, and H.~Wilton}, {\em Decision problems for
  3-manifolds and their fundamental groups}, in Baykur, R. Inanc (ed.) et al.,
  Interactions between low dimensional topology and mapping class groups.
  Proceedings of the conference, Bonn, Germany, July 1-5, 2013., vol.~19 of
  Geometry and Topology Monographs, 2015, pp.~201--236.

\bibitem{Annex-Neubueser}
{\sc C.~M. Campbell, G.~Havas, and E.~F. Robertson}, {\em Addendum to: ``{A}n
  elementary introduction to coset table methods in computational group
  theory'' by {J}. {N}eub\"user}, in Groups---{S}t. {A}ndrews 1981, vol.~71 of
  London Math. Soc. Lecture Note Ser., Cambridge Univ. Press, Cambridge, 2007,
  pp.~361--364.

\bibitem{CLO}
{\sc D.~Cox, J.~Little, and D.~O'Shea}, {\em Ideals, varieties, and
  algorithms}, Undergraduate Texts in Mathematics, Springer, New York,
  third~ed., 2007.
\newblock An introduction to computational algebraic geometry and commutative
  algebra.

\bibitem{Fomenko-Matveev}
{\sc A.~T. Fomenko and S.~V. Matveev}, {\em Algorithmic and computer methods
  for three-manifolds}, vol.~425 of Mathematics and its Applications, Kluwer
  Academic Publishers, Dordrecht, 1997.
\newblock Translated from the 1991 Russian original by M. Tsaplina and Michiel
  Hazewinkel and revised by the authors, With a preface by Hazewinkel.

\bibitem{Garey-Johnson}
{\sc M.~R. Garey and D.~S. Johnson}, {\em Computers and intractability}, W. H.
  Freeman and Co., San Francisco, Calif., 1979.
\newblock A guide to the theory of NP-completeness, A Series of Books in the
  Mathematical Sciences.

\bibitem{GMW}
{\sc D.~{Groves}, J.~{Fox Manning}, and H.~{Wilton}}, {\em {Recognizing
  geometric 3-manifold groups using the word problem}}.
\newblock arXiv:1210.2101, 2012.

\bibitem{Hempel}
{\sc J.~Hempel}, {\em Residual finiteness for {$3$}-manifolds}, in
  Combinatorial group theory and topology ({A}lta, {U}tah, 1984), vol.~111 of
  Ann. of Math. Stud., Princeton Univ. Press, Princeton, NJ, 1987,
  pp.~379--396.

\bibitem{Koiran}
{\sc P.~Koiran}, {\em Randomized and deterministic algorithms for the dimension
  of algebraic varieties}, in {Proceedings of the 1997 IEEE Conference on
  Foundations of Computer Science}, IEEE, 1997, pp.~36--45.

\bibitem{Kreuzer-Robbiano}
{\sc M.~Kreuzer and L.~Robbiano}, {\em Computational commutative algebra. 2},
  Springer-Verlag, Berlin, 2005.

\bibitem{LM85}
{\sc A.~Lubotzky and A.~R. Magid}, {\em Varieties of representations of
  finitely generated groups}, Mem. Amer. Math. Soc., 58 (1985), pp.~xi+117.

\bibitem{Miller}
{\sc C.~F. Miller, III}, {\em Decision problems for groups---survey and
  reflections}, in Algorithms and classification in combinatorial group theory
  ({B}erkeley, {CA}, 1989), vol.~23 of Math. Sci. Res. Inst. Publ., Springer,
  New York, 1992, pp.~1--59.

\bibitem{Neubueser}
{\sc J.~Neub{\"u}ser}, {\em An elementary introduction to coset table methods
  in computational group theory}, in Groups---{S}t. {A}ndrews 1981 ({S}t.
  {A}ndrews, 1981), vol.~71 of London Math. Soc. Lecture Note Ser., Cambridge
  Univ. Press, Cambridge-New York, 1982, pp.~1--45.

\bibitem{Newstead1978}
{\sc P.~E. Newstead}, {\em Introduction to moduli problems and orbit spaces},
  vol.~51 of Tata Institute of Fundamental Research Lectures on Mathematics and
  Physics, Tata Institute of Fundamental Research, Bombay; by the Narosa
  Publishing House, New Delhi, 1978.

\bibitem{Rabin}
{\sc M.~O. Rabin}, {\em Recursive unsolvability of group theoretic problems},
  Ann. of Math. (2), 67 (1958), pp.~172--194.

\bibitem{Rubinstein}
{\sc J.~H. Rubinstein}, {\em An algorithm to recognize the {$3$}-sphere}, in
  Proceedings of the {I}nternational {C}ongress of {M}athematicians, {V}ol.\ 1,
  2 ({Z}\"urich, 1994), Birkh\"auser, Basel, 1995, pp.~601--611.

\bibitem{Schleimer}
{\sc S.~Schleimer}, {\em Sphere recognition lies in {NP}}, in Low-dimensional
  and symplectic topology, vol.~82 of Proc. Sympos. Pure Math., Amer. Math.
  Soc., Providence, RI, 2011, pp.~183--213.

\bibitem{Thompson}
{\sc A.~Thompson}, {\em Thin position and the recognition problem for {$S\sp
  3$}}, Math. Res. Lett., 1 (1994), pp.~613--630.

\bibitem{Zentner}
{\sc R.~{Zentner}}, {\em {Integer homology 3-spheres admit irreducible
  representations in $\text{\em SL}(2,\mathbb{C})$}}.
\newblock arXiv:1605.08530, 2016.

\end{thebibliography}
\end{document}